\RequirePackage[l2tabu, orthodox]{nag}

\documentclass[preprint,12pt]{elsarticle}
\usepackage{amsmath}
\usepackage{amsthm}
\usepackage{amssymb}
\usepackage[colorlinks=true]{hyperref}

\journal{J. Approx. Theory}
\newtheorem{theorem}{Theorem}

\newtheorem{corollary}[theorem]{Corollary}

\newtheorem{remark}[theorem]{Remark}

\numberwithin{equation}{section}

\newcommand{\doi}[1]{\; \href{https://doi.org/#1}{doi:~#1}}

\def\Ai{\mathrm {Ai}}
\def\Bi{\mathrm {Bi}}
\def\ph{\mathrm {ph}\hspace{0.3mm}}
\renewcommand\O[1]{\mathcal{O}\left({#1}\right)}

\renewcommand\Re{\hspace{0.2mm}\mathrm{Re}\hspace{0.2mm}}
\renewcommand\Im{\hspace{0.2mm}\mathrm{Im}\hspace{0.2mm}}

\begin{document}

\begin{frontmatter}

\title{Asymptotic approximations of the continuous Hahn polynomials and their zeros}

\author[SZU]{Li-Hua Cao}
\ead{macaolh@szu.edu.cn}

\author[CUHKSZ]{Yu-Tian Li\corref{cor}}
\ead{liyutian@cuhk.edu.cn}

\author[SCUT]{Yu Lin}
\ead{scyulin@scut.edu.cn}

\address[SZU]{Department of Mathematics,
Shenzhen University,
Shenzhen, Guangdong, China
}

\address[CUHKSZ]{
School of Science and Engineering,
Chinese University of Hong Kong, Shenzhen,
Shenzhen, Guangdong, China
}

\address[SCUT]{
Department of Mathematics,
South China University of Technology,
Guangzhou, Guangdong, China
}

\cortext[cor]{Corresponding author.}

\begin{abstract}
Asymptotic approximations for the continuous Hahn polynomials and their zeros as the degree grows to infinity are established via their three-term recurrence relation.
The methods are based on the uniform asymptotic expansions for difference equations developed by Wang and Wong (\textit{Numer. Math.}, 2003) and
the matching technique in the complex plane developed by Wang (\textit{J. Approx. Theory}, 2014).

\end{abstract}

\end{frontmatter}

\section{Introduction}
The continuous Hahn polynomials are defined by
\begin{align}
p_n(x)&:=p_n(x;a,b,c,d)
\label{eq:pn}
\\
&=i^n\frac{(a+c)_n(a+d)_n}{n!}
{}_3F_2\left(\left.\begin{array}{c}
-n,n+a+b+c+d-1,a+ix\\
a+c,a+d\end{array}\right|1
\right),
\nonumber
\end{align}
where the generalized hypergeometric function is
\[
{}_3F_2\left(\left.\begin{array}{c}
\alpha,\beta,\gamma\\
\delta,\eta\end{array}\right|z
\right)
=\sum_{n=0}^\infty\frac{(\alpha)_n(\beta)_n(\gamma)_n}{(\delta)_n(\eta)_n}\frac{z^n}{n!}
\]
with $(\alpha)_n=\Gamma(n+\alpha)/\Gamma(\alpha)$ denoting the Pochhammer symbol.
The above series expansion of ${}_3F_2$ is generically convergent for $|z|<1$ and for $z=1$ for some cases;
while for the polynomial $p_n(x)$, the series in (\ref{eq:pn}) is terminating.
If the real parts of $a,b,c,d$ are all positive, $c=\overline{a}$ and $d=\overline{b}$,
then $\{p_n(x)\}_{n=0}^\infty$ form a sequence of orthogonal polynomials (see~\cite{Askey1985,Koelink}):
\begin{equation}\label{eq:orthogonal}
\frac{1}{2\pi}\int_{-\infty}^\infty p_m(x)p_n(x)w(x)dx=h_n\delta_{mn},
\end{equation}
where the weight function is
\begin{equation}\label{eq:weight}
w(x)=\Gamma(a+ix)\Gamma(b+ix)\Gamma(c-ix)\Gamma(d-ix)=\left|\Gamma(a+ix)\Gamma(b+ix)\right|^2,
\end{equation}
and
\begin{equation}\label{eq:hn}
h_n= \frac{\Gamma(n+a+c)\Gamma(n+a+d)\Gamma(n+b+c)\Gamma(n+b+d)}{(2n+a+b+c+d-1)\Gamma(n+a+b+c+d-1)n!}.
\end{equation}

In the Askey scheme of hypergeometric orthogonal polynomials, the continuous Hahn polynomials are below the Wilson polynomials
and above the Jacobi polynomials and the Meixner-Pollaczek polynomials,
since the continuous Hahn polynomials provide one extra degree of freedom
more than the Jacobi polynomials or the Meixner-Pollaczek polynomials,
and there are limiting relations between them.
In 1982, Askey and Wilson~\cite{AW} found what nowadays known as the symmetric continuous Hahn polynomials which are orthogonal on $\mathbb R$
with respect to the weight function $|\Gamma(\alpha+i x)\Gamma(\gamma+i x)|^2$
(provided that $\alpha,\gamma>0$ or $\alpha=\bar\gamma$ and $\Re\alpha>0$),
these polynomials have two parameters.
Later, Atakishiyev and Suslov~\cite{AS} introduced the continuous Hahn polynomials which have one extra parameter;
see also the comments in Askey~\cite{Askey1985}, where the orthogonality in~(\ref{eq:orthogonal}) was given.
Koelink~\cite{Koelink} observed that the Fourier transform of the Jacobi polynomials leads to the continuous Hahn polynomials,
and he obtained the orthogonality in (\ref{eq:orthogonal}) based on this fact.
A nice survey of the history of the continuous Hahn polynomials is also provided in~\cite{Koelink}.
Among all the facts, we mention that the Bateman polynomials and a sequence of polynomials of Touchard are special cases of
the symmetric continuous Hahn polynomials.

In 1933, Bateman~\cite{Bateman} introduced a sequence of polynomials $F_n(x)$ which is given by
\begin{equation}
F_n(x)={}_3F_2\left(\left.\begin{array}{c}
-n,n+1,\frac{1+x}{2}\\
1,1\end{array}\right|1
\right).
\end{equation}
Several years later, Pasternack~\cite{Pasternack} generalized the Bateman polynomials to a broader set $F_n^{\,\, m}(x)$:
\begin{equation}
F_n^{\,\, m}(x)={}_3F_2\left(\left.\begin{array}{c}
-n,n+1,\frac{1+m+x}{2}\\
1,m+1\end{array}\right|1
\right).
\end{equation}
The orthogonality of these polynomials was derived by Carlitz~\cite{Carlitz1959}.
Clearly, the Bateman polynomials $F_n(x)$ and their generalizations, the Pasternack-Bateman polynomials $F_n^{\,\, m}(x)$, are  the
continuous Hahn polynomials with the parameters, $a,b,c,d$, taking special values.

A sequence of polynomials $Q_n(x)$ was introduced by Touchard~\cite{Touchard} in the study of Taylor's expansion of the exponential function $e^{x(e^z-1)}$.
The explicit expression of these polynomials was later given by Wyman and Moser~\cite{WM}.
It was Carlitz~\cite{Carlitz1957} who noticed that the polynomials of Touchard $Q_n(x)$ are related to the Bateman polynomials:
\begin{equation}
Q_n(x)=(-1)^n2^n n!\binom{2n}{n}^{-1}F_n(2x+1).
\end{equation}

In the paper~\cite{AS} where Atakishiyev and Suslov introduced the continuous Hahn polynomials
(however in another name, the Hahn polynomials of an imaginary argument),
the authors discussed the close relationship between these polynomials and unitary irreducible representations of the Lorentz group $SU(3,1)$.
The continuous Hahn polynomials can also be interpreted as the Clebsch-Gordan coefficients (or the 3-$j$ symbols) for the Lie algebra
$\mathfrak{su}(1,1)$; see  Koelink and Van der Jeugt~\cite{KVdJ}.
These polynomials also appeared in Suslov and Trey's treatment on nonrelativistic and relativistic Coulomb problems recently; see~\cite{ST}.
The Touchard polynomials $Q_n(x)$ and Pasternack-Bateman polynomials $F_n^{\,\, \lambda}(x)$
are used in the Hermite-Pad\'e approximations of the exponential functions and
in Pr\'evost's proof of the irrationality of the $\zeta(2)$ and $\zeta(3)$, where $\zeta(x)$ is the Riemann zeta function; see~\cite{Prevost,PR}.
A connection between the polynomials of $Q_n(x)$ and the combinatorial problem of set partitions
can be found in~\cite{Rota,SunWu}.

An important issue in the study of orthogonal polynomials is to establish
their asymptotic behavior as the degree grows to infinity.
There are several ways to derive the asymptotics for a sequence of orthogonal polynomials,
namely, the method of steepest descent for integrals,
the WKB approximation for differential equations
and the nonlinear steepest descent method for Riemann--Hilbert problems.
For details of these methods, we refer the reader to
Wong~\cite{Wong1}, Olver~\cite{Olver} and Deift \textit{et al.}~\cite{Deift}.
However, to the best of our knowledge,
there is no existing result on the asymptotic approximations of the continuous Hahn polynomials.
The difficulty in applying the above-mentioned approaches to the continuous Hahn polynomials
is that these polynomials are too complicated due to the number of parameters.
For instance, there are no known differential equations that these polynomials satisfy;
the integral representations of these polynomials are complicated,
the generating functions are given as the product of two ${}_1F_1$ functions
or a single ${}_3F_2$ function:
\begin{equation}
{}_1F_1\left(\left.\begin{array}{c}a+ix\\ a+c\end{array}\right|-it\right)
{}_1F_1\left(\left.\begin{array}{c}d-ix\\ b+d\end{array}\right|it\right)
=\sum_{n=0}^\infty\frac{p_n(x;a,b,c,d)}{(a+c)_n(a+d)_n}t^n,
\end{equation}
\begin{equation}
{}_1F_1\left(\left.\begin{array}{c}a+ix\\ a+d\end{array}\right|-it\right)
{}_1F_1\left(\left.\begin{array}{c}c-ix\\ b+c\end{array}\right|it\right)
=\sum_{n=0}^\infty\frac{p_n(x;a,b,c,d)}{(a+d)_n(b+c)_n}t^n,
\end{equation}
\begin{equation}
\begin{aligned}
(1-t)^{1-a-b-c-d}&\hspace{1mm}
{}_3F_2\left(\left.\begin{array}{c}\frac{a+b+c+d-1}{2},\frac{a+b+c+d}{2},a+ix\\ a+c,a+d\end{array}\right|-\frac{4t}{(1-t)^2}\right)\\
=&\sum_{n=0}^\infty\frac{(a+b+c+d-1)_n}{(a+c)_n(a+d)_n}p_n(x;a,b,c,d)t^n;
\end{aligned}
\end{equation}
see~\cite[eqs.~(9.4.11)-(9.4.13)]{KLS}.
These formulas imply that the integral representations of the continuous Hahn polynomials
are too complicated and difficult to analyze.
Moreover, the weight function given in (\ref{eq:weight}) is also complicated and consists of four gamma functions,
which implies that the weight function has poles on the complex $x$-plane.
The loss of analyticity of the weight function suggests it is difficult, if not impossible, to apply the powerful Riemann-Hilbert approach.
To make the zeros of the continuous Hahn polynomials lie in a fixed interval, one introduces a rescaling $x=nt$,
which makes the poles of the weight function approach the real-line in the complex $t$-plane,
and the nonlinear steepest descent method (the contour deformation)
for the Riemann--Hilbert problem cannot be applied in such a case.
To avoid this difficulty,
one may let some or all of the parameters vary and depend on the degree $n$,
which will set the poles of the weight function far away from the real axis.
This, however, changes the problem in the sense that the polynomials under considered
are orthogonal with respect to a varying weight and differ from the original ones.
A similar treatment on the Meixner--Pollaczek polynomials was carried out in~\cite{WQW}.
Note that the continuous Hahn polynomials are in the level above
the Meixner--Pollaczek polynomials in the Askey scheme,
and the weight function of the former is more complicated than the one of the latter,
one can imagine that how tedious it would be when applying the Riemann--Hilbert approach
to the continuous Hahn polynomials even for the varying weight problem.

Each sequence of orthogonal polynomials satisfies a three-term recurrence relation.
Hence, a natural approach to derive the asymptotics of orthogonal polynomials is to analyze the recurrence relations.
A lot of investigations have been done under this direction and it turns out to be a great success by the effort of many authors.
Among the complete list, we briefly mention \cite{CL,CC,DHW,DM1,DM2,Geronimo,Li,VAG,WangXSWong,WangZWong1,WangZWong2, WongLi1992a, WongLi1992b}.
See also an application of this approach to some polynomials from indeterminate moment problems~\cite{DIW} and a survey article on this approach~\cite{Wong2}.

In this paper, we shall apply the method introduced in~\cite{WangXSWong,WangZWong1} to derive the asymptotic approximations for the continuous Hahn polynomials.
This is based on the fact that the three-term recurrence relation
of the continuous Hahn polynomials is simple,
which could be derived straightforwardly from the ${}_3F_2$ polynomials
\begin{equation}
\widetilde{p}_n(x)=\frac{n!}{i^n(a+c)_n(a+d)_n}p_n(x;a,b,c,d),
\end{equation}
and it reads
\begin{equation}
(a+ix)\widetilde{p}_n(x)=A_n\widetilde{p}_{n+1}(x)-(A_n+C_n)\widetilde{p}_n(x)+C_n\widetilde{p}_{n-1}(x),
\end{equation}
where
\begin{equation}
{A}_n=-\frac{(n+a+b+c+d-1)(n+ a+c)(n+a+d)}{(2n+a+b+c+d-1)(2n+a+b+c+d)},
\end{equation}
\begin{equation}
{C}_n=\frac{n(n+ b+c-1)(n+b+d-1)}{(2n+a+b+c+d-2)(2n+a+b+c+d-1)};
\end{equation}
see~\cite{KLS} and \cite[\S 18.19]{NIST}.
The monic continuous Hahn polynomials $\pi_n(x)$ are given by
\begin{equation}\label{eq:monic}
\begin{aligned}
\pi_n(x)=&\frac{n!}{(n+a+b+c+d-1)_n} p_n(x;a,b,c,d)\\
=&i^n \frac{(a+c)_n(a+d)_n}{(n+a+b+c+d-1)_n}\widetilde{p}_n(x),
\end{aligned}
\end{equation}
and satisfy the following three-term recurrence relation
\begin{equation}\label{eq:rec-monic}
x\pi_n(x)=\pi_{n+1}(x)+i(A_n+C_n+a)\pi_n(x)-A_{n-1}C_n \pi_{n-1}(x)
\end{equation}
with $\pi_{-1}(x)=0$ and $\pi_0(x)=1$; see~\cite{KLS}.
Note that,  $A_n+C_n+a$ is purely imaginary provided that $c=\overline{a}$ and $d=\overline{b}$.

\section{Statement of the main results}
\label{sec:results}
The main results in this paper are the asymptotic approximations of the continuous Hahn polynomials,
which include nonuniform asymptotics both in the zero-free region and oscillatory region,
uniform asymptotic expansions around the turning points, and Plancherel-type asymptotics near the turning points.
Asymptotic formulas of the zeros of these polynomials are also included.

An application of Theorem~2 in Ismail and Li~\cite{IL}, see also Theorem~1.1 in~\cite{DIW},
shows the zeros of the continuous Hahn polynomials $\pi_n(x)$ lie asymptotically in the interval $(-n/2,n/2)$.
Therefore, we introduce a scaling $x=nt$, then a neighborhood of the interval $(-1/2,1/2)$
in the complex $t$-plane will be referred to as the oscillatory region,
and the region for $t\in \mathbb C\setminus[-1/2,1/2]$ is called the non-oscillatory region (or zero-free region).
The two values of $t$, $t_+=1/2$ and $t_-=-1/2$, are the turning points.

The first result here is the asymptotic formulas for $\pi_n(x)$ in three different regions of the $t$-plane.
Thought this paper, the branch cut of a multi-valued function, such as $t^{1/2}$, $t^{1/4}$ and $\log t$,
is taken as the negative real axis.
For the sake of convenience, we set
\begin{equation}\label{eq:uv}
u:=\frac{a+b+c+d}{2},\qquad v:=\frac{a+b-c-d}{2i}.
\end{equation}

\begin{theorem}[Non-uniform asymptotic approximations]
\label{thm:nonuniform}
The monic continuous Hahn polynomials have the following large-$n$ asymptotic approximations:
\begin{equation}\label{eq:nonuniform_1}
\begin{aligned}
\pi_n(nt)
=
&2^{\frac32-2u}\left(\frac{n}{4e}\right)^n\frac{t^{1-u}}{(4t^2-1)^{1/4}}
\exp\left\{(2nt+v)\arcsin\frac{1}{2t}\right.
\\
&+\left.\Big(n+u-\frac12\Big)\log\big[2t+(4t^2-1)^{1/2}\big]\right\}
\times\left\{1+\O{\frac{1}{n}}\right\}
\end{aligned}
\end{equation}
for $t$ in any compact subset of $\mathbb C\setminus[-1/2,1/2]$; and
\begin{align}
\pi_n(nt)
\sim
&2^{\frac52-2u}\left(\frac{n}{4e}\right)^n\frac{t^{1-u}}{(1-4t^2)^{1/4}}\exp\left\{(2nt+v)\frac\pi2\right\}
\nonumber
\\
&\times
\left\{\cos\left\{(2nt+v)\log\left[\frac{1+(1-4t^2)^{1/2}}{2t}\right]\right.\right.
\label{eq:nonuniform_2}
\\
&\hspace{12mm}\left.\left.+i\Big(n+u-\frac12\Big)\log\left[2t+i(1-4t^2)^{1/2}\right]+\frac{\pi}{4}\right\}
+\O{\frac{1}{n}}\right\}
\nonumber
\end{align}

for $t$ in any compact subset of $\mathbb C\setminus \big\{(-\infty,0]\cup[\frac12,\infty)\big\}$; and
\begin{align}
\pi_n(nt)
\sim& 2^{\frac52-2u}\left(\frac {n}{4e}\right)^n
\frac{(-t)^{1-u}}{(1-4t^2)^{1/4}}\exp\left\{-(2nt+v)\frac\pi2\right\}
\nonumber
\\
&\times
\left\{\cos\left\{(2nt+v)\log\left[\frac{1+(1-4t^2)^{1/2}}{-2t}\right]\right.\right.
\\
&\left.\left.+i\Big(n+u-\frac12\Big)\log\left[2t+i(1-4t^2)^{1/2}\right]+\frac{4u-3}{4}\pi\right\}
+\O{\frac1n}\right\}
\nonumber
\end{align}
for $t$ in any compact subset of $\mathbb C\setminus \big\{(-\infty,-\frac12]\cup[0,\infty)\big\}$.

\end{theorem}

Around each of the turning points, there is a uniform asymptotic expansion in terms of the Airy function $\Ai(\cdot)$ and its derivative.
In the following result, we shall use the notation:
\begin{equation}\label{eq:Kn1}
K_n:=\frac{\Gamma(\frac{n+1}{2})\Gamma(\frac{n+2u-1}{2})\Gamma(\frac{n+a+c}{2})\Gamma(\frac{n+a+d}{2})\Gamma(\frac{n+b+c}{2})\Gamma(\frac{n+b+d}{2})}
{2^n\,\Gamma(\frac{n+u-1/2}{2})\Gamma(\frac{n+u}{2})^2\Gamma(\frac{n+u+1/2}{2})}.
\end{equation}

\begin{theorem}[Uniform asymptotic expansions]
\label{thm:uniform}

Let
\begin{equation}\label{eq:zeta1}
\left\{\begin{aligned}
&\frac23[-\zeta(t)]^\frac32
=
\arccos(2t)-2t\log\frac{1+\sqrt{1-4t^2}}{2t},
& 0<t<\frac12,
\\
&\frac23[\zeta(t)]^\frac32
=
\pi t-2t\arcsin\left(\frac1{2t}\right)
-\log(2t+\sqrt{4t^2-1}),
\,\, &t\geq\frac12,
\end{aligned}\right.
\end{equation}
and
\begin{equation}\label{eq:_Phi1}
\Phi(\zeta)=\frac{v}{\zeta^\frac12}\left[\frac\pi2-\arcsin\frac1{2t}-\log(2t+\sqrt{4t^2-1})\right].
\end{equation}
Then the monic continuous Hahn polynomials have the following asymptotic expansions for large $n$:

\noindent
(i) around the turning point $t_+=1/2$,
\begin{equation}\label{eq:uniform_1}
\begin{aligned}
&\left[w(\nu t)\right]^\frac12\frac{\pi_n(\nu t)}{K_n}
\\
\sim
&2^{1-u}\sqrt{2\pi}\left(\frac{\zeta}{4t^2-1}\right)^\frac14
\\
&\times
\left[\Ai\left(\nu ^\frac23\zeta+\nu ^{-\frac13}\Phi\right)\sum_{s=0}^\infty\frac{A_s(\zeta)}{\nu ^{s-\frac16}}\right.
+\left.\Ai'\left(\nu ^\frac23\zeta+\nu ^{-\frac13}\Phi\right)\sum_{s=0}^\infty\frac{B_s(\zeta)}{\nu ^{s+\frac16}}
\right]
\end{aligned}
\end{equation}
for $0<t<\infty$, where $w(x)$ is the weight function and
\begin{equation}\label{eq:nu}
\nu=n+u-v-\frac12;
\end{equation}

\noindent
(ii) around the turning point $t_-=-1/2$,
\begin{equation}\label{eq:uniform_2}
\begin{aligned}
&\left[w(-\mu t)\right]^\frac12\frac{\pi_n(-\mu t)}{K_n}
\\
\sim
& (-1)^n2^{1-u}\sqrt{2\pi}\left(\frac{\zeta}{4t^2-1}\right)^\frac14
\left[\Ai\left(\mu ^\frac23\zeta-\mu ^{-\frac13}\Phi\right)\sum_{s=0}^\infty\frac{A_s(\zeta)}{\mu ^{s-\frac16}}\right.
\\
&\hspace{50mm}+\left.\Ai'\left(\mu ^\frac23\zeta-\mu ^{-\frac13}\Phi\right)\sum_{s=0}^\infty\frac{B_s(\zeta)}{\mu ^{s+\frac16}}
\right]
\end{aligned}
\end{equation}
for $0<t<\infty$, where $w(x)$ is the weight function and
\begin{equation}\label{eq:mu}
\mu=n+u+v-\frac12.
\end{equation}
In the above expansions, the coefficients $A_0(\zeta)\equiv1$, $B_0(\zeta)\equiv0$ and
all the other coefficients $A_s(\zeta)$, $B_s(\zeta)$ are determined in a recursive manner.
\end{theorem}

\begin{corollary}[Plancherel-type asymptotics]
\label{cor:Plancherel}

(i) Let $x=\nu/2+4^{-2/3}\nu^{1/3}s$ with $\nu$ given in (\ref{eq:nu}).
Then,
\begin{equation}\label{eq:Plancherel_1}
\begin{aligned}
\left[w(x)\right]^\frac12\frac{\pi_n(x)}{K_n}
=2^{\frac43-u}\sqrt{\pi}
\left[\Ai(s)+\O{\nu^{-1}}\right]
\end{aligned}
\end{equation}
as $n\to\infty$ for $s$ in any bounded real interval.

(ii) Let $x=-\mu/2-4^{-2/3}\mu^{1/3}s$ with $\mu$ given in (\ref{eq:mu}). Then,
\begin{equation}\label{eq:Plancherel_2}
\begin{aligned}
\left[w(x)\right]^\frac12\frac{\pi_n(x)}{K_n}
= (-1)^n2^{\frac43-u}\sqrt{\pi}
\left[\Ai(s)+\O{\mu^{-1}}\right]
\end{aligned}
\end{equation}
as $n\to\infty$ for $s$ in any bounded real interval.

\end{corollary}

\begin{corollary}[Asymptotics of zeros]
\label{cor:zeros}
Let $x_{n,k}$ denote the zeros of the continuous Hahn polynomials in ascending order, $x_{n,1}<x_{n,2}<\cdots<x_{n,n}$,
and let $\mathfrak{a}_k$ be the zeros of the Airy function $\Ai(x)$ in descending order. Then,
the largest zeros and smallest zeros of the continuous Hahn polynomials have the following asymptotic approximations
\begin{equation}\label{eq:zeros_1}
x_{n,n-k}=\frac12\nu+4^{-\frac23}\nu^{\frac13}\mathfrak{a}_k+\O{\nu^{-2/3}}
\end{equation}
and
\begin{equation}\label{eq:zeros_2}
x_{n,k}=-\frac12\mu-4^{-\frac23}\mu^{\frac13}\mathfrak{a}_k+\O{\mu^{-2/3}}
\end{equation}
as $n\to\infty$ for bounded $k$, respectively.
Here, $\nu$ and $\mu$ are given in (\ref{eq:nu}) and (\ref{eq:mu}) respectively.
\end{corollary}

\begin{remark}
(i) The orthogonality (\ref{eq:orthogonal}) for the continuous Hahn polynomials
holds when $\Re(a,b,c,d)>0$.
If $c=\overline{a}$ and $d=\overline{b}$, or if $c=\overline{b}$ and $d=\overline{a}$
then the weight function $w(x)$ is positive, and $p_n(x)$ is orthogonal on the real-line.

(ii) For the general case, if $(a+c)$, $(a+d)$, $(b+c)$ and $(b+d)$
are not equal to zero or negative integers,
the orthogonality (\ref{eq:orthogonal}) still holds
with the integration contour deformed to the complex $x$-plane
to separate the increasing sequences of poles from the decreasing sequences of poles;
see~\cite{Askey1985}.

(iii) The results on asymptotics of the continuous Hahn polynomials stated above
are established for case (i).
Theorems~\ref{thm:nonuniform} and~\ref{thm:uniform}
are still valid for the case (ii),
when the quantities $u$ and $v$ given in~(\ref{eq:uv}) are complex numbers.
The zeros of the polynomials are complex valued in the general case,
and hence Corollary~\ref{cor:zeros} should be amended accordingly in such a situation.

(iv)
All the asymptotic formulas in this paper are not valid for $t$ in a neighbourhood of the origin.
This might be seen from the branch cuts that we take for the multi-valued functions,
and a reason for that is $t=0$, under the scaling $x=nt$,
is an accumulation point of the poles of the weight function, cf.~(\ref{eq:weight}).
For the asymptotic approximation of the continuous Hahn polynomials when the variable $x$ is fixed,
see a recent paper~\cite{HaidariLi}.

\end{remark}

\section{Nonuniform asymptotic approximations}
\label{sec:nonuniform}

\begin{proof}[Proof of Theorem~\ref{thm:nonuniform}]

\

Here, we shall use Lemma 2.1 of Dai-Ismail-Wang~\cite{DIW}
to derive the asymptotic approximation in (\ref{eq:nonuniform_1}).
Following the notations in~\cite{DIW} and with a bit of calculation,
the two coefficients of the recurrence relation (\ref{eq:rec-monic}) are
\[
\begin{aligned}
&a_n:=i(A_n+C_n+a)=i\frac{v}{2}+\O{1},
\\
&b_n:=-A_{n-1}C_n=\frac{n^2}{16}+\frac{u-1}{8}n+\O{1},
\end{aligned}
\]
where $u$ and $v$ are defined in (\ref{eq:uv}).
Since the letter $a$ is already used in the polynomials as a parameter,
we slightly change the notations used in~\cite[Lemma 2.1]{DIW},
namely,
\[
\begin{aligned}
&a_n=\alpha_0 n^p+\alpha_1 n^{p-1}+\O{n^{p-2}},
\\
&b_n=\beta_0^2 n^{2p}+\beta_1 n^{2p-1}+\O{n^{2p-2}}.
\end{aligned}
\]
Now, it follows that
\begin{equation}
p=1,\quad
\alpha_0=0,\quad
\alpha_1=i\frac{v}{2},\quad
\beta_0=\frac14,\quad
\beta_1=\frac{u-1}{8}.
\end{equation}
Take $\sigma=0$ and set $x=nt$. The interval $I$ is now given by
$I=[\alpha_0-2\beta_0,\alpha_0+2\beta_0]=[-1/2,1/2]$.
Then a direct application of \cite[Lemma 2.1]{DIW} gives
\[
\begin{aligned}
\pi_n(nt)
=& \left(\frac n2\right)^n\left[\frac{2t+\sqrt{4t^2-1}}{4t}\right]^{1/2}
\exp\left\{n\int_0^1\log\frac{2t+\sqrt{4t^2-r^2}}{2}\,dr\right\}
\\
&\times\exp\left\{\frac12\int_0^1\frac{r\,dr}{4t^2-r^2}-iv\int_0^1\frac{dr}{\sqrt{4t^2-r^2}}\right\}
\\
&\times\exp\left\{(1-u)\int_0^1\frac{r\,dr}{\sqrt{4t^2-r^2}[2t+\sqrt{4t^2-r^2}]}\right\}
\times\left\{1+\O{\frac{1}{n}}\right\}.
\end{aligned}
\]
A slight calculation shows that
\[
\begin{aligned}
\pi_n(x)
=\left(\frac {n}{4e}\right)^n&\frac{2^{\frac32-2u} t^{1-u}}{(4t^2-1)^{1/4}}
\exp\left\{(2nt+v)\arcsin\frac{1}{2t}\right.\\
&\left.+\Big(n+u-\frac12\Big)\log\left[2t+(4t^2-1)^{1/2}\right]\right\}
\times\left\{1+\O{\frac{1}{n}}\right\}
\end{aligned}
\]
as $n\to\infty$.
Note that the last equation holds uniformly for $t$ in any compact subset of $\mathbb C\setminus[-\frac12,\frac12]$.

\

To derive the asymptotic approximation of the continuous Hahn polynomials in the oscillatory region,
we shall follow the approach developed recently by Wang \cite{WangXS}. The idea of Wang~\cite{WangXS} is the analytic continuation
for the multi-valued functions such as the square root and the logarithm.
Denote
\[
\begin{aligned}
\varphi_n(t)=&2^{\frac32-2u}\left(\frac {n}{4e}\right)^n\exp\left\{(1-u)\log t-\frac14\log(2t-1)-\frac14\log(2t+1)\right\}
\\
&\times\exp\left\{(2nt+v)\arcsin\frac{1}{2t}+\Big(n+u-\frac12\Big)\log\left(2t+(4t^2-1)^{1/2}\right)\right\}.
\end{aligned}
\]
Note that $\varphi_n(t)$ is analytic for $t\in\mathbb C\setminus[-\frac12,\frac12]$,
$t=-\frac12$, $t=\frac12$ and $t=0$ are branch points,
and $t=0$ is also an essential singularity of $\varphi_n(t)$.
By choosing different branch cuts, we can define two functions $\varphi_n^+(t)$ and $\varphi_n^{-}(t)$,
which are equal to $\varphi_n(t)$ for $t$ in the upper and lower complex plane, respectively.
To be precise,
\begin{equation}\label{eq:phip}
\begin{aligned}
\varphi_n^{+}(t)=&2^{\frac32-2u}\left(\frac {n}{4e}\right)^n\exp\left\{(1-u)\log t-\frac14\log(1-4t^2)-i\frac{\pi}{4}\right\}
\\
&\times\exp\left\{(2nt+v)\left[\frac\pi2-i\log\Big(\frac{1+(1-4t^2)^{1/2}}{2t}\Big)\right]\right\}\\
&\times\exp\left\{\Big(n+u-\frac12\Big)\log\left[2t+i(1-4t^2)^{1/2}\right]\right\}
\end{aligned}
\end{equation}
and
\begin{equation}\label{eq:phin}
\begin{aligned}
\varphi_n^{-}(t)=&2^{\frac32-2u}\left(\frac {n}{4e}\right)^n\exp\left\{(1-u)\log t-\frac14\log(1-4t^2)+i\frac{\pi}{4}\right\}
\\
&\times\exp\left\{(2nt+v)\left[\frac\pi2+i\log\Big(\frac{1+(1-4t^2)^{1/2}}{2t}\Big)\right]\right\}\\
&\times\exp\left\{-\Big(n+u-\frac12\Big)\log\left[2t+i(1-4t^2)^{1/2}\right]\right\}.
\end{aligned}
\end{equation}
Note that
\[
\arcsin(z)=-i\log\big[(1-z^2)^{1/2}+iz\big],\qquad  z\in\mathbb C\setminus\{(-\infty,-1)\cup(1,\infty)\};
\]
see \cite[(4.13.19)]{NIST}. Therefore,
\[
\begin{aligned}
\arcsin\left(\frac{1}{2t}\right)
=&-i\log\left[\Big(1-\frac{1}{4t^2}\Big)^{1/2}+i\frac{1}{2t}\right]\\
=&\frac\pi2-i\log\left[\frac{1+(1-4t^2)^{1/2}}{2t}\right]
\end{aligned}
\]
for $\Im t>0$. Similarly, one has
\[
\begin{aligned}
\arcsin\left(\frac{1}{2t}\right)
=&\frac\pi2-i\log\left[\frac{1-(1-4t^2)^{1/2}}{2t}\right]
\end{aligned}
\]
for $\Im t<0$.
Moreover,  $\Re \big(t\arcsin\frac1{2t}\big)>0$ for $\Im t>0$ and $\Re \big(t\arcsin\frac1{2t}\big)<0$ for $\Im t<0$.
The following facts are readily seen:

(i) $\varphi_n^{+}(t)+\varphi_n^{-}(t)$ is analytic on $\mathbb C\setminus\{(-\infty,0]\cup[\frac12,\infty)\}$;

(ii) $\varphi_n^{-}(t)/\varphi_n^{+}(t)$ is exponentially small for large $n$ when $\Im t>0$;

(iii)  $\varphi_n^{+}(t)/\varphi_n^{-}(t)$ is exponentially small for large $n$ when $\Im t<0$.

\

\noindent
Therefore,
\[
\begin{aligned}
\pi_n(nt)\sim&\varphi_n^{+}(t)+\varphi_n^{-}(t)\\
\sim& 2^{\frac52-2u}\left(\frac {n}{4e}\right)^n\exp\left\{(1-u)\log t-\frac14\log(1-4t^2)+(2nt+v)\frac\pi2\right\}
\\
&\times
\cos\left\{(2nt+v)\log\Big[\frac{1+(1-4t^2)^{1/2}}{2t}\Big]\right.\\
&\hspace{1.5cm}\left.+i\Big(n+u-\frac12\Big)\log\left[2t+i(1-4t^2)^{1/2}\right]+\frac{\pi}{4}\right\}
\end{aligned}
\]
for $t$ in $(0,1/2)$. This result can be extended to $t$ in compact subsets of
$\mathbb{C}\setminus\{(-\infty,0]\cup[1/2,\infty)\}$.

To obtain a large-$n$ asymptotic approximation of $\pi_n(nt)$
for $t$ in a region containing the interval $(-1/2,0)$,
we choose the branch cut of the logarithm function in (\ref{eq:phip}) and (\ref{eq:phin})
so that $\log t=\log(-t)+\pi i$ for $t<0$.
A similar argument as above leads to
\[
\begin{aligned}
\pi_n(nt)
\sim& 2^{\frac52-2u}\left(\frac {n}{4e}\right)^n\exp\left\{(1-u)\log(-t)-\frac14\log(1-4t^2)-(2nt+v)\frac\pi2\right\}
\\
&\times
\cos\left\{(2nt+v)\log\Big[\frac{1+(1-4t^2)^{1/2}}{-2t}\Big]\right.\\
&\hspace{1.5cm}\left.+i\Big(n+u-\frac12\Big)\log\left[2t+i(1-4t^2)^{1/2}\right]+\frac{\pi}{4}-(1-u)\pi\right\}
\end{aligned}
\]
for $t$ in $(-1/2,0)$. This result can be extended to $t$ in compact subsets of
$\mathbb{C}\setminus\{(-\infty,-1/2]\cup[0,\infty)\}$.
This completes the proof of Theorem~\ref{thm:nonuniform}.
\end{proof}

\section{Uniform asymptotic expansions}
\label{sec:uniform}

\begin{proof}[Proof of Theorem~\ref{thm:uniform}]

\

Following Wang and Wong~\cite{WangZWong1,WangZWong2}, we define a sequence $\{K_n\}$ by
\begin{equation}\label{eq:Kn}
K_n:=\frac{\Gamma(\frac{n+1}{2})\Gamma(\frac{n+2u-1}{2})\Gamma(\frac{n+a+c}{2})\Gamma(\frac{n+a+d}{2})\Gamma(\frac{n+b+c}{2})\Gamma(\frac{n+b+d}{2})}
{2^n\,\Gamma(\frac{n+u-1/2}{2})\Gamma(\frac{n+u}{2})^2\Gamma(\frac{n+u+1/2}{2})}.
\end{equation}
It is readily seen that
\[
\frac{K_{n+1}}{K_{n-1}}=-A_{n-1}C_n.
\]
Denote
\begin{equation}
P_n(x):=\frac{\pi_n(x)}{K_n},
\quad
\overline{A}_n:=\frac{K_n}{K_{n+1}},
\quad
\overline{B}_n:=-i(A_n+C_n+a)\frac{K_n}{K_{n+1}}.
\end{equation}
It follows from (\ref{eq:rec-monic}) that $P_n(x)$ satisfy the following recurrence relations
\begin{equation}\label{eq:rec_P}
P_{n+1}(x)+P_{n-1}(x)=(\overline{A}_nx+\overline{B}_n)P_n(x),
\end{equation}
where
\begin{equation}\label{eq:barAnBn}
\overline{A}_n\sim n^{-1}\sum_{s=0}^{\infty}\frac{\alpha_s}{n^s}
\qquad\text{and}\qquad
\overline{B}_n\sim\sum_{s=0}^\infty\frac{\beta_s}{n^s}
\end{equation}
with the first few coefficients given by
\begin{equation}
\alpha_0=4,
\quad
\alpha_1=2-4u,
\quad
\beta_0=0,
\quad
\beta_1=2v.
\end{equation}
In terms of the notions of Wang and Wong~\cite{WangZWong1,WangZWong2},
$\theta=1$ and the two turning points $t_{\pm}$ are
\[
\alpha_0 t_\pm+\beta_0=\pm 2,
\]
that is
\begin{equation}
t_-=-\frac12, \qquad
t_+=\frac 12.
\end{equation}
Hence, the recurrence relation in (\ref{eq:rec_P}) belongs to the case $\theta\neq0$ by the classification in~\cite{WangZWong1,WangZWong2,CL}, see also~\cite{Wong2}.

We first concentrate on the turning point $t_+=1/2$.
Following Wang and Wong~\cite{WangZWong1}, we set $\nu:=n+\tau_0$ with
\begin{equation}
\tau_0=-\frac{\alpha_1 t_++\beta_1}{(2-\beta_0)\theta}=u-v-\frac12.
\end{equation}
Then, the expansions of the recurrence coefficients in (\ref{eq:barAnBn}) can be rewritten as
\begin{equation}\label{eq:barAnBnnew}
\overline{A}_n\sim \nu^{-\theta}\sum_{s=0}^{\infty}\frac{\alpha_s'}{\nu^s}
\qquad\text{and}\qquad
\overline{B}_n\sim\sum_{s=0}^\infty\frac{\beta_s'}{\nu^s}
\end{equation}
with $\theta=1$. The first few coefficients of the above expansions are
\begin{equation}\label{eq:ab}
\alpha_0'=4,\qquad \beta_0'=0,\qquad \alpha_1'=-4v,\qquad\beta_1'=2v.
\end{equation}
By the main result in Wang and Wong~\cite{WangZWong1},
two linearly independent solutions of (\ref{eq:rec_P}) are given by
\begin{equation}\label{eq:solution_P}
\begin{aligned}
\mathcal{P}_n(x)
\sim\left(\frac{\zeta}{4t^2-1}\right)^\frac14
&\left[
\Ai\left(\nu ^\frac23\zeta+\nu ^{-\frac13}\Phi\right)
\sum_{s=0}^\infty\frac{A_s(\zeta)}{\nu ^{s-\frac16}}
\right.
\\
&\hspace{3mm}+\left.
\Ai'\left(\nu ^\frac23\zeta+\nu ^{-\frac13}\Phi\right)
\sum_{s=0}^\infty\frac{B_s(\zeta)}{\nu ^{s+\frac16}}
\right]
\end{aligned}
\end{equation}
and
\begin{equation}\label{eq:solution_Q}
\begin{aligned}
\mathcal{Q}_n(x)
\sim\left(\frac{\zeta}{4t^2-1}\right)^\frac14
&\left[
\Bi\left(\nu ^\frac23\zeta+\nu ^{-\frac13}\Phi\right)
\sum_{s=0}^\infty\frac{A_s(\zeta)}{\nu ^{s-\frac16}}
\right.
\\
&\hspace{3mm}+\left.
\Bi'\left(\nu ^\frac23\zeta+\nu ^{-\frac13}\Phi\right)
\sum_{s=0}^\infty\frac{B_s(\zeta)}{\nu ^{s+\frac16}}
\right],
\end{aligned}
\end{equation}
as $n\to\infty$ for $t>0$, where
the function $\zeta(t)$ plays the role of the Liouville-Green transformation
defined as
\[
\left\{
\begin{alignedat}{2}
&\frac23[-\zeta(t)]^\frac32
=-\alpha_0' t^\frac{1}{\theta}\int_t^\frac12\frac{s^{-1/\theta}}{\sqrt{4-(\alpha_0's+\beta_0')^2}}ds
+\arccos\frac{\alpha_0't+\beta_0'}{2}
\\
&\qquad\qquad \text{for}\quad 0<t<\frac12,
\\
&\frac23[\zeta(t)]^\frac32
=\alpha_0' t^{\frac{1}{\theta}}\int_{\frac12}^t
\frac{s^{-1/\theta}}{\sqrt{(\alpha_0's+\beta_0')^2-4}}ds
-\log\frac{\alpha_0't+\beta_0'+\sqrt{(\alpha_0't+\beta_0')^2-4}}{2}
\\
&\qquad\qquad \text{for}\quad t\geq\frac12,
\end{alignedat}
\right.
\]
and $\Phi(\zeta)$ is defined as
\[
\Phi(\zeta)=\frac{1}{\zeta^\frac12}\int_{t_+}^t\frac{\alpha_1'\tau+\beta_1'}{\theta\tau\sqrt{(\alpha_0'\tau+\beta_0')^2-4}}d\tau;
\]
see \cite[eqs. (4.10) and (4.28)]{WangZWong1}.
Recall that $\theta=1$ and the quantities given in (\ref{eq:ab}),
a slight calculation yields
\begin{equation}\label{eq:zeta}
\left\{\begin{alignedat}{2}
&\frac23[\zeta(t)]^\frac32=
\pi t-2t\arcsin\left(\frac1{2t}\right)-\log(2t+\sqrt{4t^2-1}),
&\quad t\geq\frac12,\\
&\frac23[-\zeta(t)]^\frac32=
\arccos(2t)-2t\log\frac{1+\sqrt{1-4t^2}}{2t},
& 0<t<\frac12,\\
\end{alignedat}\right.
\end{equation}
and
\begin{equation}\label{eq:_Phi}
\Phi(\zeta)=\frac{v}{\zeta^\frac12}\left[\frac\pi2-\arcsin\frac1{2t}-\log(2t+\sqrt{4t^2-1})\right].
\end{equation}
The coefficients in the expansions of (\ref{eq:solution_P})-(\ref{eq:solution_Q})
are given by $A_0(\zeta)\equiv1$, $B_0(\zeta)\equiv0$
and all the other coefficients can be derived in a recursive manner.

Note that $P_n(x)$ is a linear combination of the two solutions given in (\ref{eq:solution_P}) and (\ref{eq:solution_Q}), that is
\begin{equation}\label{eq:linearcom}
[w(x)]^\frac12P_n(x)=C_1(x)\mathcal{P}_n(x)+C_2(x)\mathcal{Q}_n(x),
\end{equation}
where $C_1(x)$ and $C_2(x)$ are two functions that depend on $x$ only,
$w(x)$ is the weight function given in (\ref{eq:weight}).
To determine the coefficients in this linear combination,
we notice that (\ref{eq:linearcom}) holds true for complex $x$ by analytic continuation.
Therefore, there is an overlapping region in which both (\ref{eq:linearcom})
and the nonuniform asymptotic approximation in (\ref{eq:nonuniform_1}) are valid.
And hence $C_1(x)$ and $C_2(x)$ can be determined by matching these two results in the overlapping region.

In order to match the results in (\ref{eq:solution_P})--(\ref{eq:linearcom}) and (\ref{eq:nonuniform_1}),
we shall rewrite these results first.
Recall the asymptotics of the Airy functions for the large argument
\[
\Ai(z)\sim \frac{1}{2\sqrt{\pi}z^{1/4}}e^{-\frac23 z^\frac32},
\qquad
\Bi(z)\sim  \frac{1}{\sqrt{\pi}z^{1/4}}e^{\frac23 z^\frac32},
\qquad |\ph z|\leq \frac13\pi-\delta,
\]
for some any $\delta>0$; see~\cite[\S 9.7]{NIST}.
Then, we have
\begin{equation}\label{eq:asymP}
\begin{aligned}
\mathcal{P}_n(x)\sim
\frac{1}{2\sqrt{\pi}}\left(\frac1{4t^2-1}\right)^\frac14
\exp\bigg\{& (2\nu t+v)\Bigl[\arcsin\Big(\frac1{2t}\Big)-\frac{\pi}{2}\Bigr]
\\
&+(\nu+v)\log(2t+\sqrt{4t^2-1})\bigg\}
\end{aligned}
\end{equation}
and
\begin{equation}\label{eq:asymQ}
\begin{aligned}
\mathcal{Q}_n(x)\sim
\frac{1}{\sqrt{\pi}}\left(\frac1{4t^2-1}\right)^\frac14
\exp\bigg\{&-(2\nu t+v)\Bigl[\arcsin\Big(\frac1{2t}\Big)-\frac{\pi}{2}\Bigr]
\\
&-(\nu+v)\log(2t+\sqrt{4t^2-1})\bigg\}
\end{aligned}
\end{equation}
as $\nu\to\infty$ for $t>0$. Here, we have also made use the fact that
\[
\begin{aligned}
&\frac23\left(\nu^\frac23\zeta+\nu^{-\frac13}\Phi\right)^\frac32
\\
=&\frac23\nu\zeta^\frac32\left[1+\frac32\frac{\Phi}{\nu\zeta}+\O{\nu^{-2}}\right]
\\
=&\frac23\nu\zeta^\frac32+v\left[\frac\pi2-\arcsin\frac1{2t}-\log(2t+\sqrt{4t^2-1})\right]
+\O{\nu^{-1}}.
\end{aligned}
\]
From Stirling's formula, we have
\[
|\Gamma(a+i \nu  t)|
\sim\sqrt{2\pi}\exp\left\{(\Re a-1/2)\log|\nu  t|-\frac\pi2|\nu  t+\Im a|\right\},
\]
as $\nu t\to\infty$; see~\cite[eq. (5.11.9)]{NIST}. Hence, when $t>0$
\begin{equation}\label{eq:w-asy}
\begin{aligned}
{}[w(x)]^\frac12
=&|\Gamma(a+i \nu t)\Gamma(b+i \nu t)|
\\
\sim& 2\pi(\nu t)^{u-1}e^{-\pi \nu t-\pi\frac v2},\qquad \text{as\,\,}\nu\to\infty.
\end{aligned}
\end{equation}
The above asymptotic formula holds true for $t$
in any compact subset of $\mathbb C\setminus(-\infty,0]$ by analytic continuation.
Recall that
\[
P_n(x)=\frac{\pi_n(x)}{K_n}.
\]
By Stirling's formula and the definition of $K_n$ in (\ref{eq:Kn}), we obtain
\[
\frac{1}{K_n}\sim \frac{1}{2\pi}\left(\frac{4e}{n}\right)^n\left(\frac{2}{n}\right)^{u-1},
\qquad\text{as}\quad n\to\infty.
\]
Note that $\nu=n+u-v-1/2$. Then, we have from (\ref{eq:nonuniform_1}) that as $n\to\infty$
\[
\begin{aligned}
\pi_n(\nu t)\sim &
2^{\frac32-2u}\left(\frac {n}{4e}\right)^nt^{1-u}\left(\frac{1}{4t^2-1}\right)^\frac{1}{4}
\\
&\times
\exp\left\{(2\nu t+v)\arcsin\frac{1}{2t}+(\nu+v)\log(2t+\sqrt{4t^2-1})\right\}.
\end{aligned}
\]
A combination of the last four formulas leads to
\[
\begin{aligned}
&\left[w(x)\right]^\frac12P_n(x)
\\
\sim&e^{-\pi \nu t-\frac v2 \pi}2^{\frac12-u}\left(\frac{1}{4t^2-1}\right)^\frac{1}{4}
\\
&\times
\exp\left\{(2\nu t+v)\arcsin\frac{1}{2t}+\Big(n+u-\frac12\Big)\log(2t+\sqrt{4t^2-1})\right\},
\end{aligned}
\]
as $n\to\infty$ and for $t>\frac12$.
Comparing the last formula with (\ref{eq:linearcom})--(\ref{eq:asymQ}) leads to
\begin{equation}
C_1(x)\equiv 2^{1-u}\sqrt{2\pi},
\qquad
C_2(x)\equiv 0.
\end{equation}
This completes the proof of the first part of Theorem~\ref{thm:uniform}.

\

The second part of Theorem~\ref{thm:uniform} can be proved in a similar manner.
At the turning point $t_-=-1/2$, we set
\begin{equation}\label{eq:P-hat}
\widehat{P}_n(x)=(-1)^nP_n(-x).
\end{equation}
The turning point $t_-=-1/2$ becomes $1/2$ under this transformation.
It follows from (\ref{eq:rec_P}) and (\ref{eq:P-hat}) that
$\{\widehat{P}_n(x)\}$ satisfy the following three-term recurrence relation
\begin{equation}\label{eq:rec-Phat}
\widehat{P}_{n+1}(x)+\widehat{P}_{n-1}(x)=(\overline{A}_nx-\overline{B}_n)\widehat{P}_n(x),
\end{equation}
which only differs from (\ref{eq:rec_P}) in the sign of $\overline{B}_n$.
Let
\begin{equation}
\mu:=n+\widehat{\tau_0}=n-\frac{\alpha_1/2 -\beta_1}{(2+\beta_0)\theta}=n+u-v-\frac12.
\end{equation}
There are two linearly independent solutions of (\ref{eq:rec-Phat}) which
have the forms in (\ref{eq:solution_P})-(\ref{eq:solution_Q}).
Comparing (\ref{eq:rec-Phat}) with (\ref{eq:rec_P}),
one just needs to replace the function $\Phi(\zeta)$ by
\begin{equation}
\begin{aligned}
\widehat\Phi(\zeta)
=\frac{v}{\zeta^\frac12}\left[\arcsin\frac1{2t}-\frac\pi2+\log\left(2t+\sqrt{4t^2-1}\right)\right]
=-\Phi(\zeta).
\end{aligned}
\end{equation}
Then, a similar matching technique as we have done in the first part yields
\[
\begin{aligned}
&\left[w(-\mu t)\right]^\frac12 \widehat{P}_n(\mu t)
\\
\sim&2^{1-u}\sqrt{2\pi}\left(\frac{\zeta}{4t^2-1}\right)^\frac14
\\
&\times\left[
\Ai\left(\mu ^\frac23\zeta-\mu ^{-\frac13}\Phi\right)\sum_{s=0}^\infty\frac{A_s(\zeta)}{\mu ^{s-\frac16}}
+\Ai'\left(\mu ^\frac23\zeta-\mu ^{-\frac13}\Phi\right)\sum_{s=0}^\infty\frac{B_s(\zeta)}{\mu ^{s+\frac16}}
\right],
\end{aligned}
\]
as $n\to\infty$ for $t>0$.
This completes the proof of Theorem~\ref{thm:uniform}.
\end{proof}

\

\begin{proof}[Proof of Corollary~\ref{cor:Plancherel}]

\

Note that the function $\zeta(t)$ in (\ref{eq:zeta}) is monotonically increasing and has the asymptotic behavior
\[
\zeta(t)=4^{2/3}\left(t-\frac12\right)+\O{\Big(t-\frac12\Big)^2},
\qquad
\text{as}\quad t\to\frac12.
\]
Moreover,
\[
\frac{\Phi(\zeta)}{\zeta}\sim -\frac{1}{3}+\O{t-\frac12},
\qquad
\text{as}\quad t\to\frac12.
\]
Take $t=1/2+(4\nu)^{-2/3}s$ with $s$ in an arbitrary bounded interval.
Then,
\[
\nu^{\frac23}\zeta(t)+\nu^{-\frac13}\Phi(\zeta)=s+\O{\nu^{-1}},
\qquad
\text{as}\quad \nu\to\infty.
\]
Notice that $A_0(\zeta)\equiv1$ and $B_0(\zeta)\equiv0$ in (\ref{eq:uniform_1}).
Then the Plancherel-type asymptotic formula (\ref{eq:Plancherel_1}) is a direct consequence of (\ref{eq:uniform_1}).
In a similar manner, the Plancherel-type asymptotic formula in (\ref{eq:Plancherel_2})
follows from (\ref{eq:uniform_2}) by taking $t=1/2+(4\mu)^{-2/3}s$.
\end{proof}

For the asymptotic formula of the zeros of the continuous Hahn polynomials,
we recall the following result of Hethcote~\cite{Hethcote}:

\begin{theorem}[Hethcote]
In the interval $[a-\rho,a+\rho]$, suppose that $f(t)=g(t)+\varepsilon(t)$,
where $f(t)$ is continuous, $g(t)$ is differentiable,
$g(a)=0$, $m:=\min|g'(t)|>0$ and
\[
E=\max|\varepsilon(t)|<\min\big\{|g(a-\rho)|,|g(a+\rho)|\big\}.
\]
Then, there exists a zero $c$ of $f(t)$ in the interval such that $|c-a|\leq E/m$.
\end{theorem}

Applying this result to (\ref{eq:Plancherel_1}) immediately leads to the formula (\ref{eq:zeros_1}).
The formula in (\ref{eq:zeros_2}) can be established in a similar manner.

\section*{Acknowledgements}
All authors would like to thank the two Reviewers' invaluable comments,
their suggestions and comments improve the paper a lot and make it more readable.
This work was supported in part by the National Natural Science Foundation of China [grant numbers 11571375, 11501215, and 11801480],
the Research Grants Council of Hong Kong SAR [grant numbers 12303515 and 12328416],
and the President's Fund from the Chinese University of Hong Kong, Shenzhen.

\end{document}